\documentclass[portuges,12pt,letter]{article}
\usepackage[centertags]{amsmath}
\usepackage{amsfonts}
\usepackage{newlfont}
\usepackage{amscd}
\usepackage{graphics}
\usepackage{epsfig}
\usepackage{indentfirst}
\usepackage{amsxtra}
\usepackage[latin1]{inputenc}
\usepackage{amssymb, amsmath}
\usepackage{amsthm}
\usepackage[mathscr]{eucal}

\newtheorem{thm}{Theorem}[section]

\newtheorem{obe}[thm]{Remark}

\title{\bf{A duality principle for non-linear elasticity}}

\author{Fabio Silva Botelho }

\begin{document}
\maketitle

\begin{abstract}
This article develops a duality principle for non-linear elasticity. The results are obtained through standard tools of convex analysis and the
Legendre transform concept.

We emphasize the dual variational formulation obtained is concave. Moreover, sufficient optimality conditions are also established. \\
\end{abstract}
\section{Introduction}
At this point we start to describe the primal formulation.

Consider $\Omega \subset \mathbb{R}^3$ an open, bounded, connected set,
which represents the reference volume of an elastic solid
under the loads $f \in L^2(\Omega;\mathbb{R}^3)$ and the boundary loads $\hat{f} \in L^2(\Gamma;\mathbb{R}^3)$, where $\Gamma$  denotes
the boundary of $\Omega$. The field of displacements resulting from the actions
of $f$  and $\hat{f}$ is denoted by $u \equiv (u_1,u_2,u_3) \in U$, where
$u_1,u_2,$ and $u_3$ denotes the displacements relating the
directions $x, y,$ and $z$ respectively, in the cartesian system
$(x,y,z)$.

Here $U$ is defined by
\begin{equation}
U=\{u=(u_1,u_2,u_3) \in W^{1,4}(\Omega;\mathbb{R}^3) \; | \;
u=(0,0,0)\equiv \mathbf{0}  \text{ on } \Gamma_0\}
\end{equation}
and $\Gamma=\Gamma_0 \cup \Gamma_1$, $\Gamma_0 \cap \Gamma_1=
\emptyset$  (for details about the Sobolev space $U$ see \cite{1}). We assume $|\Gamma_0|>0$ where $|\Gamma_0|$ denotes the Lebesgue measure of $\Gamma_0.$

The stress tensor is denoted by $\{\sigma_{ij}\}$, where
\begin{gather}
\sigma_{ij}=H_{ijkl}\left(\frac{1}{2}(u_{k,l}+u_{l,k}
+u_{m,k}u_{m,l})\right),
\end{gather}
 $$\{H_{ijkl}\}=\{\lambda \delta_{ij} \delta_{kl}+\mu(\delta_{ik}\delta_{jl}+\delta_{il}\delta_{jk})\},$$ $\{\delta_{ij}\}$ is the Kronecker delta and $\lambda,\mu>0$ are the Lam\'{e} constants (we assume they are such that $\{H_{ijkl}\}$ is a symmetric constant  positive definite forth order tensor).
 Here, $i,j,k,l \in \{1,2,3\}.$

 The boundary value form of the non-linear
elasticity model is given by
\begin{gather}\label{9.9.10.1}
\left \{
\begin{array}{ll}
 \sigma_{ij,j}+(\sigma_{mj}u_{i,m})_{,j}+f_i=0, &  \text{ in } \Omega,
 \\
  u= \mathbf{0}, & \text{ on }\Gamma_0,
\\
\sigma_{ij}n_j+\sigma_{mj}u_{i,m}n_j=\hat{f}_i, & \text{ on } \Gamma_1,
  \end{array} \right.\end{gather}
 where $\textbf{n}=(n_1,n_2,n_3)$ denotes the outward normal to the surface $\Gamma.$

The corresponding primal variational formulation is represented by
$J:U \rightarrow \mathbb{R}$, where
\begin{eqnarray}
J(u)&=&\frac{1}{2}\int_\Omega H_{ijkl}\left(\frac{1}{2}(u_{i,j}+u_{j,i}
+u_{m,i}u_{m,j})\right)\left(\frac{1}{2}(u_{k,l}+u_{l,k}+
u_{m,k}u_{m,l})\right)dx\nonumber \\ &&-\langle u,f \rangle_{L^2(\Omega;\mathbb{R}^3)}-\int_{\Gamma_1} \hat{f}_i u_i \;d\Gamma
\end{eqnarray}
where $$\langle u,f \rangle_{L^2(\Omega;\mathbb{R}^3)}=\int_\Omega f_i u_i \;dx.$$
\begin{obe} Derivatives must be always understood in the distributional
sense, whereas boundary conditions are in the sense of traces. Moreover, from now on by a regular Lipschitzian boundary $\Gamma$ of $\Omega$ we mean regularity enough so that the standard Gauss-Green formulas of integrations by parts and the well known Sobolev imbedding and trace theorems to hold.
Also, we denote by $\mathbf{0}$ the zero vector in  appropriate function spaces, the standard norm for
$L^2(\Omega)$ by $\|\cdot\|_2$ and $L^2(\Omega;\mathbb{R}^{3\times 3})$ simply by $L^2$.

About the references, we refer to \cite{10,11,2900,85} as the first articles to deal with the convex analysis approach applied to non-convex and non-linear mechanics models. Indeed, the present work complements such important original publications, since in these previous results the complementary energy is established as a perfect duality principle for the case of positive definiteness of the stress tensor (or the membrane force tensor, for plates and shells models) at a critical point.

We have relaxed such constraints, allowing to some extent, the stress tensor to not  be necessarily either positive or negative definite in $\Omega$.
Similar problems and models are addressed in \cite{120}.

Moreover, existence results for models in elasticity are addressed in \cite{903,3,4}. Finally, the standard tools of convex analysis here used may be found in
\cite{6,12,29,120}.
\end{obe}

\section{The main duality principle}

At this point, in order to clarify the notation,  we recall a tensor $M=\{M_{ijkl}\}$ is said to be positive definite,  if there exists $c_0>0$ such that $$M_{ijkl}t_{ij}t_{kl}\geq c_0\; t_{ij}t_{ij}, \forall t \in \mathbb{R}^{3\times 3},
$$ and in such a case we denote $M>\mathbf{0}$. Similarly, for appropriate tensors $M_1,M_2$ of this type, we shall denote $M_1>M_2$ if $M_1-M_2>\mathbf{0}$.

The main duality principle is summarized by the following theorem.
\begin{thm} Let $\Omega \subset \mathbb{R}^3$ be an open, bounded, connected set with a regular (Lipschitzian) boundary
denoted by $\Gamma=\Gamma_0 \cup \Gamma_1,$ where $\Gamma_0 \cap \Gamma_1=\emptyset,$ and $|\Gamma_0|>0.$

Consider the functional $(G \circ \Lambda):U \rightarrow \mathbb{R}$ defined by
$$(G \circ \Lambda)(u)=\frac{1}{2}\int_\Omega H_{ijkl} \left(\frac{u_{i,j}+u_{j,i}}{2}+\frac{u_{m,i}u_{m,j}}{2} \right)\left(\frac{u_{k,l}+u_{l,k}}{2}+\frac{u_{m,k}u_{m,l}}{2} \right)\;dx,$$
where $\Lambda:U \rightarrow Y \times Y$ is given by,
$$\Lambda u=\{\Lambda_1 u, \Lambda_2 u\},$$
$$\Lambda_1(u)=\left\{ \frac{u_{i,j}+u_{j,i}}{2}\right\},$$
$$\Lambda_2 u=\{u_{m,i}\}.$$

Here,
\begin{eqnarray} U&=&\{ u \in W^{1,4}(\Omega;\mathbb{R}^3)\;:\;
\nonumber \\ &&  u=(u_1,u_2,u_3)=(0,0,0)=\mathbf{0}, \text{ on } \Gamma_0\},
\end{eqnarray}
and
$$Y=Y^*=L^2(\Omega; \mathbb{R}^{3 \times 3})\equiv L^2.$$

Define $(F\circ \Lambda_2):U \rightarrow \mathbb{R}$ by
$$(F\circ \Lambda_2)(u)=\frac{K}{2}\langle u_{m,i}, u_{m,i}\rangle_{L^2},$$
and
$(G_K \circ \Lambda):U \rightarrow \mathbb{R}$ by
$$G_K(\Lambda u)=G(\Lambda u)+\frac{K}{2}\langle u_{m,i}, u_{m,i}\rangle_{L^2}.$$

Also, define
$$C=\{u \in U\;:\; (G_K)^{**}(\Lambda u)=G_K(\Lambda u)\},$$
where $K>0$ is a constant such that
$$M=\left\{ \frac{D_{ijkl}}{2K}-\overline{H}_{ijkl}\right\}$$ is a positive definite tensor.

Here  \begin{equation}D_{ijkl}=\left\{\begin{array}{ll}
1,& \text{ if } i=k \text{ and } j=l, \\
0,& \text{ otherwise }\end{array} \right. \end{equation}
and, in an appropriate sense,
$$\{\overline{H}_{ijkl}\}=\{H_{ijkl}\}^{-1}.$$

Moreover, for $f \in L^2(\Omega;\mathbb{R}^3)$, $\hat{f} \in L^2(\Gamma_1;\mathbb{R}^3)$, let $J:U \rightarrow \mathbb{R}$ be defined by,
$$J(u)=G(\Lambda u)-\langle u_i,f_i \rangle_{L^2(\Omega)}-\langle u_i,f_i \rangle_{L^2(\Gamma_1)}.$$

Under such hypotheses,
$$\inf_{u \in U} J(u) \geq \sup_{(Q,\tilde{\sigma}) \in A^*} \tilde{J}^*(Q , \tilde{\sigma}),$$
where,
$$\tilde{J}^*(Q,\tilde{\sigma})=\inf_{z^* \in B^*(Q,\tilde{\sigma})} J^*(Q,\tilde{\sigma},z^*),$$

$$J^*(Q,\tilde{\sigma},z^*)=F^*(z^*)-G_K^*(Q,\tilde{\sigma},z^*),$$

\begin{eqnarray}F^*(z^*)&=&\sup_{v_2 \in Y} \{\langle (z^*)_{ij},(v_2)_{ij}\rangle_{L^2}-F(\{(v_2)_{ij}\})
\nonumber \\ &=& \frac{1}{2K}\int_\Omega z^*_{ij} z^*_{ij}\;dx,\end{eqnarray}

\begin{eqnarray}
G_K^*(Q,\tilde{\sigma},z^*)&=& \sup_{ (v_1,v_2) \in Y \times Y}\{\langle z^*_{ij},(v_1)_{ij} \rangle_{L^2}+
\langle \tilde{\sigma}_{ij},(v_1)_{ij}\rangle_{L^2} \nonumber \\ &&+\langle Q_{mi}, (v_2)_{mi} \rangle_{L^2}
-G_K(\{(v_1)_{ij}\},\{(v_2)_{ij}\})\} \nonumber \\ &=&
\sup_{(v_1,v_2) \in Y \times Y} \{\langle z^*_{ij},(v_1)_{ij} \rangle_{L^2}+
\langle \tilde{\sigma}_{ij},(v_1)_{ij}\rangle_{L^2} \nonumber \\ &&+\langle Q_{mi}, (v_2)_{mi} \rangle_{L^2}
 \nonumber \\ &&-\frac{1}{2}\int_\Omega H_{ijkl} \left((v_1)_{ij}+\frac{(v_2)_{mi} (v_2)_{mj}}{2} \right)
 \nonumber \\ && \times \left( (v_1)_{kl}+\frac{(v_2)_{mk}(v_2)_{ml}}{2} \right)\;dx \nonumber \\ &&- \frac{K}{2}\langle (v_2)_{mi}, (v_2)_{mi}\rangle_{L^2}\}
 \nonumber \\ &=& \frac{1}{2}\int_\Omega \overline{z^*_{ij}+\tilde{\sigma}_{ij}+K \delta_{ij}}\;Q_{mi}Q_{mj}\;dx \nonumber \\ &&
 +\frac{1}{2}\int_\Omega \overline{H}_{ijkl} (\tilde{\sigma}_{ij}+z^*_{ij})(\tilde{\sigma}_{kl} +z^*_{kl})\;dx
 \end{eqnarray}
 if $\{z^*_{ij}+\tilde{\sigma}_{ij}+K \delta_{ij}\}$ is positive definite, where
 $$\{z^*_{ij}+\tilde{\sigma}_{ij}+K\delta_{ij}\}=\left[\begin{array}{lcr}
z_{11}^*+\tilde{\sigma}_{11}+K & z_{12}^*+\tilde{\sigma}_{12} & z_{13}^*+\tilde{\sigma}_{13} \\
z_{21}^*+\tilde{\sigma}_{21} & z_{22}^*+\tilde{\sigma}_{22} +K& z_{23}^*+\tilde{\sigma}_{23} \\
z_{31}^*+\tilde{\sigma}_{31} & z_{32}^*+\tilde{\sigma}_{32} & z_{33}^*+\tilde{\sigma}_{33}+K,
\end{array} \right] $$
$$\overline{\{z^*_{ij}+\tilde{\sigma}_{ij}+K \delta_{ij}\}}=\{z_{ij}^*+\tilde{\sigma}_{ij}+K \delta_{ij}\}^{-1}.$$

Also, $B^*(\tilde{\sigma})=B_1(\tilde{\sigma})  \cap B_2$,
\begin{eqnarray}B_1(\tilde{\sigma})&=&\{z^* \in Y^*\;:\;
\nonumber \\ && \{z_{ij}^*+\tilde{\sigma}_{ij}+K \delta_{ij}\} \geq \{K\delta_{ij}/2\}, \text{ in } \Omega\},\end{eqnarray}
$$B_2=\{z^* \in Y^*\;:\; |z^*_{ij}|<K/8,\;  \text{ in } \Omega,\;\forall i,j \in \{1,2,3\}\},$$
and
$$C_1=\{u \in U\;:\; |u_{i,j}|< 1/8, \text{ in } \Omega,\; \forall i,j \in \{1,2,3\}\}.$$

Moreover,

$A^*=A_1 \cap A_2 \cap A_3 \cap A_4,$ where
\begin{eqnarray}
A_1&=&\{(Q,\tilde{\sigma}) \in Y^* \times Y^* \;:\; \nonumber \\ &&
\sup_{z^* \in B^*(\tilde{\sigma})}\{ \langle z^*_{ij},u_{i,j} \rangle_{L^2}-F^*(z^*)\}=(F \circ \Lambda_2)(u),
\nonumber \\ && \forall u \in C_1 \cap C\}.
\end{eqnarray}
\begin{eqnarray}
A_2&=&\{(Q,\tilde{\sigma}_{ij}) \in Y^* \times Y^*\;:\;
\nonumber \\ && \tilde{\sigma}_{ij,j}+Q_{ij,j}+f_i=0, \; \text{ in } \Omega\}
\end{eqnarray}
\begin{eqnarray}
A_3&=&\{(Q,\tilde{\sigma}_{ij}) \in Y^* \times Y^*\;:\;
\nonumber \\ && \tilde{\sigma}_{ij}n_j+Q_{ij}n_j-\hat{f}_i=0, \; \text{ on } \Gamma_1\},
\end{eqnarray}
\begin{eqnarray}A_4&=&\left\{(Q,\tilde{\sigma}) \in Y^* \times Y^*\;:\;
 |\tilde{\sigma}_{ij}|< K/8,\;|Q_{ij}|< \frac{3K}{32}, \text{ in } \Omega,\; \forall i,j \in \{1,2,3\}\right\}.\end{eqnarray}

Finally, suppose there exists $(u_0,(Q_0,\tilde{\sigma}_0),z_0^*) \in (C \cap C_1) \times A^* \times B^*( \tilde{\sigma}_0),$ such that
\begin{eqnarray}&&\delta\{J^*(Q_0,\tilde{\sigma}_0,z_0^*) \nonumber \\ &&
+\langle (u_0)_{i,j},(\tilde{\sigma}_0)_{ij}+(Q_0)_{ij} \rangle_{L^2}-\langle (u_0)_i,f_i \rangle_{L^2}
\nonumber \\ && -\langle (u_0)_i, f_i \rangle_{L^2(\Gamma_1)}\}=\mathbf{0}.
\end{eqnarray}

Under such hypotheses,
\begin{eqnarray}
J(u_0)&=& \min_{u \in C \cap C_1} J(u) \nonumber \\ &=& \max_{(Q, \tilde{\sigma}) \in A^*} \tilde{J}(Q, \tilde{\sigma}) \nonumber \\
&=& \tilde{J}^*(Q_0,\tilde{\sigma}_0) \nonumber \\ &=& J^*(Q_0,\tilde{\sigma}_0,z_0^*).
\end{eqnarray}
\end{thm}
\begin{proof}
Observe that
\begin{eqnarray}
G_K^*(\tilde{\sigma},Q,z^*) &\geq& \langle \tilde{\sigma}_{ij}+z^*_{ij},u_{i,j} \rangle_{L^2} \nonumber \\ &&
+\langle Q_{mi},u_{m,i} \rangle_{L^2} -G_K(\Lambda u), \end{eqnarray}
$\forall (Q,\tilde{\sigma}) \in A^*,$ $z^* \in B^*(\tilde{\sigma}),\; u \in C \cap C_1.$

Hence,
\begin{eqnarray}
G_K^*(\tilde{\sigma},Q,z^*) &\geq& \langle z^*_{ij},u_{i,j} \rangle_{L^2} \nonumber \\ &&
-\langle \tilde{\sigma}_{ij,j}+Q_{ij,j},u_i \rangle_{L^2} \nonumber \\ &&+\langle \tilde{\sigma}_{ij}n_j+Q_{ij}n_j,u_i \rangle_{L^2(\Gamma_1)} -G_K(\Lambda u)
\nonumber \\ &=& \langle z^*_{ij},u_{i,j} \rangle_{L^2}+\langle u_i,f_i \rangle_{L^2} \nonumber \\ && +\langle u_i,\hat{f}_i \rangle_{L^2(\Gamma_1)}
-G_K(\Lambda u), \end{eqnarray}
$\forall (Q,\tilde{\sigma}) \in A^*,$ $z^* \in B^*(\tilde{\sigma}),\; u \in C \cap C_1,$
so that
\begin{eqnarray}&& -F^*(z^*)+G_K^*(\tilde{\sigma},Q,z^*) \nonumber \\ &\geq& -F^*(z^*)+ \langle z^*_{ij},u_{i,j} \rangle_{L^2}+\langle u_i,f_i \rangle_{L^2} \nonumber \\ && +\langle u_i,\hat{f}_i \rangle_{L^2(\Gamma_1)}
-G_K(\Lambda u), \end{eqnarray}
$\forall (Q,\tilde{\sigma}) \in A^*,$ $z^* \in B^*(\tilde{\sigma}),\; u \in C \cap C_1.$
Therefore,
\begin{eqnarray}&& \sup_{z^* \in B^*(\tilde{\sigma})}\{-F^*(z^*)+G_K^*(\tilde{\sigma},Q,z^*)\} \nonumber \\ &\geq& \sup_{z^* \in B^*(\tilde{\sigma})}
\{-F^*(z^*)+ \langle z^*_{ij},u_{i,j} \rangle_{L^2}+\langle u_i,f_i \rangle_{L^2} \nonumber \\ && +\langle u_i,\hat{f}_i \rangle_{L^2(\Gamma_1)}
-G_K(\Lambda u)\} \nonumber \\ &=& (F \circ \Lambda_2)(u)-G_K(\Lambda u)\nonumber \\ &&+\langle u_i,f_i \rangle_{L^2} \nonumber \\ && +\langle u_i,\hat{f}_i \rangle_{L^2(\Gamma_1)} \nonumber \\ &=&- J(u), \end{eqnarray}
$\forall (Q,\tilde{\sigma}) \in A^*,$ $u \in C \cap C_1.$

Thus,
\begin{eqnarray} J(u) &\geq& \inf_{z^* \in B^*(\tilde{\sigma})}\{F^*(z^*)-G_K^*(Q,\tilde{\sigma},z^*)\}
\nonumber \\ &=& \inf_{z^* \in B^*(\tilde{\sigma})} J^*(Q,\tilde{\sigma},z^*) \nonumber \\ &=& \tilde{J}^*(Q , \tilde{\sigma}),\end{eqnarray}
$\forall (Q,\tilde{\sigma}) \in A^*,$ $u \in C \cap C_1$,
so that
\begin{equation}\label{el1} \inf_{u \in C \cap C_1} J(u) \geq \sup_{(Q,\tilde{\sigma}) \in A^*} \tilde{J}^*(Q,\tilde{\sigma}).\end{equation}

On the other hand,
$(u_0,(Q_0,\tilde{\sigma}_0),z_0^*) \in (C \cap C_1) \times A^* \times B^*( \tilde{\sigma}_0)$ is such that
\begin{eqnarray}&&\delta\{F^*(z_0^*)-G_K^*(Q_0,\tilde{\sigma}_0,z_0^*) \nonumber \\ &&
+\left\langle (u_0)_{i,j},(\tilde{\sigma}_0)_{ij}+(Q_0)_{ij} \right\rangle_{L^2}-\langle (u_0)_i,f_i \rangle_{L^2}
\nonumber \\ && -\langle (u_0)_i, f_i \rangle_{L^2(\Gamma_1)}\}=\mathbf{0}.
\end{eqnarray}
From the variation in $Q_{m_i}$ we obtain,
\begin{equation}\label{el23}
\frac{\partial G_K^*(Q_0,\tilde{\sigma}_0,z_0^*)}{\partial Q_{mi}}-(u_0)_{m,i}=0, \text{ in } \Omega,\end{equation}
so that,
\begin{equation} (u_0)_{m,i}=\overline{(z_0^*)_{ij}+(\tilde{\sigma}_0)_{ij}+K \delta_{ij}}(Q_0)_{mj}, \text{ in } \Omega, \end{equation}
and thus,
\begin{equation}\label{el3}
(Q_0)_{mi}=((z_0^*)_{ij}+(\tilde{\sigma}_0)_{ij})(u_0)_{m,j}+K(u_0)_{m,i}. \end{equation}

From the variation in $z^*_{mi}$, we obtain
\begin{equation}\label{el4}
\frac{\partial F^*(z^*_0)}{\partial z^*_{mi}}=\frac{\partial G_K^*(Q_0,\tilde{\sigma}_0,z_0^*)}{\partial z^*_{mi}}, \end{equation}
so that
\begin{eqnarray}\label{el5}
\frac{(z_0^*)_{ij}}{K}&=&- \frac{1}{2}\;(\overline{\overline{(z_0^*)_{ij}+(\tilde{\sigma}_0)_{ij}+K \delta_{ij}}})\;Q_{mi}Q_{mj}
+\overline{H_{ijkl}}((z^*_0)_{kl}+(\tilde{\sigma}_0)_{kl})
\nonumber \\ &=&-\frac{1}{2}(u_0)_{m,i}(u_0)_{m,j}+\overline{H_{ijkl}}((z^*_0)_{kl}+(\tilde{\sigma}_0)_{kl}),
\end{eqnarray}
where
$$\{\overline{\overline{(z_0^*)_{ij}+(\tilde{\sigma}_0)_{ij}+K \delta_{ij}}}\}=\{(z_0^*)_{ij}+(\tilde{\sigma}_0)_{ij}+K \delta_{ij}\}^{-2}.$$

From (\ref{el23}), (\ref{el4}) and the variation in $\tilde{\sigma}$, we obtain
\begin{eqnarray}\label{el6}
(u_0)_{i,j}&=&\frac{\partial G_K^*(Q_0, \tilde{\sigma}_0, z_0^*)}{\partial \tilde{\sigma}_{ij}}\nonumber \\ &=&\frac{\partial G_K^*(Q_0, \tilde{\sigma}_0, z_0^*)}{\partial z^*_{ij}}
\nonumber \\ &=&-\frac{1}{2}\;\overline{\overline{(z_0^*)_{ij}+(\tilde{\sigma}_0)_{ij}+K \delta_{ij}}})\;Q_{mi}Q_{mj}
+\overline{H_{ijkl}}((z^*_0)_{kl}+(\tilde{\sigma}_0)_{kl}) \nonumber \\ &=&-\frac{1}{2}(u_0)_{m,i}(u_0)_{m,j}+\overline{H_{ijkl}}((z^*_0)_{kl}+(\tilde{\sigma}_0)_{kl}),
\end{eqnarray}
so that, from this and (\ref{el5}), we have
\begin{equation}\label{el7} (z_0^*)_{ij}=K (u_0)_{i,j}, \end{equation}
Indeed from the concerning symmetries
\begin{eqnarray}\label{el8}
&&\frac{\partial G_K^*(Q_0,\tilde{\sigma}_0,z_0^*)}{\partial [(\tilde{\sigma}_{ij}+\tilde{\sigma}_{ji})/2]}
\nonumber \\ &=& -\frac{(u_0)_{m,i}(u_0)_{m,j}}{2}+\overline{H}_{ijkl}((z_0^*)_{kl}+(\tilde{\sigma}_0)_{kl}) \nonumber \\ &=&
\frac{(u_0)_{i,j}+(u_0)_{j,i}}{2}, \text{ in } \Omega. \end{eqnarray}
Thus,
\begin{eqnarray}\label{el9}
&&(\tilde{\sigma}_0)_{ij}+(z_0^*)_{ij}\nonumber \\ &=& H_{ijkl}\left(\frac{(u_0)_{k,l}+(u_0)_{l,k}}{2}+\frac{(u_0)_{m,k}(u_0)_{m,l}}{2} \right),
\text{ in } \Omega.
\end{eqnarray}

From this and (\ref{el7}) we may obtain
\begin{equation}\label{el10}
(\tilde{\sigma}_0)_{ij}= \sigma_{ij}(u_0)-K (u_0)_{i,j}
\end{equation}
where
\begin{eqnarray}\label{el11}\sigma_{ij}(u_0)&=& H_{ijkl}\left(\frac{(u_0)_{k,l}+(u_0)_{l,k}}{2}+\frac{(u_0)_{m,k}(u_0)_{m,l}}{2}\right)
\nonumber \\ &=& (\tilde{\sigma}_0)_{ij}+(z_0^*)_{ij}.
\end{eqnarray}
From this and (\ref{el3}) we have,
\begin{equation}\label{el12}
(Q_0)_{mi}=\sigma_{ij}(u_0)(u_0)_{m,j}+K (u_0)_{m,i}.
\end{equation}
From the variation in $u_0$, we obtain,
\begin{equation}\label{el13}
(\tilde{\sigma}_0)_{ij,j}+(Q_0)_{ij,j}+f_i=0, \text{ in } \Omega,
\end{equation}
and
\begin{equation}\label{el14}
(\tilde{\sigma}_0)_{ij}n_j +(Q_0)_{ij}n_j-\hat{f}_i=0, \text{ on } \Gamma_1.
\end{equation}
By (\ref{el10}), (\ref{el11}), (\ref{el12}), (\ref{el13}) and (\ref{el14}), we obtain
\begin{equation}\label{el15}
(\sigma_{ij} (u_0))_{,j}+(\sigma_{im}(u_0)(u_0)_{m,j})_{,j}+f_i=0, \text{ in } \Omega,
\end{equation}
and
\begin{equation}\label{el16}
\sigma_{ij}(u_0)n_j +\sigma_{im}(u_0)(u_0)_{m,j} n_j-\hat{f}_i=0, \text{ on } \Gamma_1.
\end{equation}
where $\sigma(u_0)$ is indicated in (\ref{el11}).

Also, from (\ref{el7})
\begin{equation}\label{el17} F^*(z_0^*)=\langle (z_0^*)_{ij}, (u_0)_{ij} \rangle_{L^2}-(F\circ \Lambda_2)(u_0). \end{equation}
From (\ref{el10})-(\ref{el14}), we may write,
\begin{eqnarray}\label{el18}
G_K^*(Q_0, \tilde{\sigma}_0,z_0^*)&=& \langle (Q_0)_{mi},(u_0)_{m,i} \rangle_{L^2}+\langle (z_0^*)_{ij},(u_0)_{i,j} \rangle_{L^2} \nonumber \\
&& +\langle (\tilde{\sigma}_0)_{ij},(u_0)_{i,j} \rangle_{L^2}-G_K(\Lambda u_0) \nonumber \\ &=&
\langle (z_0^*)_{ij},(u_0)_{i,j} \rangle_{L^2}+ \langle (u_0)_i,f_i \rangle_{L^2} \nonumber \\ && +\langle (u_0)_i, \hat{f}_i \rangle_{L^2(\Gamma_1)}
-G_K(\Lambda u_0).
\end{eqnarray}

By (\ref{el17}) and (\ref{el18}), we obtain,
\begin{eqnarray}\label{el19}
&&J^*(Q_0, \tilde{\sigma}_0, z_0^*) \nonumber \\ &=& F^*(z_0^*)-G_K^*(Q_0, \tilde{\sigma}_0,z_0^*) \nonumber \\ &=&
G_K(\Lambda u_0)-(F \circ \Lambda_2)(u_0) \nonumber \\ &&-\langle (u_0)_i,f_i \rangle_{L^2} -\langle (u_0)_i, \hat{f}_i \rangle_{L^2(\Gamma_1)}
\nonumber \\ &=& J(u_0)
\end{eqnarray}
Denoting $$\{\overline{\overline{\overline{(z^*)_{ij}+(\tilde{\sigma}_0)_{ij}+K \delta_{ij}}}}\}=\{(z^*)_{ij}+(\tilde{\sigma}_0)_{ij}+K \delta_{ij}\}^{-3}$$
and the second Fr\'{e}chet derivative of $J^*(Q,\tilde{\sigma},z^*)$ relating $z^*$ (here considering $z^*$ as an independent variable with fixed $Q$ and $\tilde{\sigma})$ at $(Q_0,\tilde{\sigma}_0,z^*)$ by
 $$\delta^2_{z^*z^*} J^*(Q_0,\tilde{\sigma}_0,z^*)$$ and also denoting \begin{gather}
(M_3)_{ijkl}=\left \{
\begin{array}{ll}
W_{ij}, &  \text{ if } i=k \text{ and } j=l,
 \\
 0, & \text{ otherwise, }
  \end{array} \right.\end{gather}
where $$W_{ij}=(\overline{\overline{\overline{(z^*)_{ij}+(\tilde{\sigma}_0)_{ij}+K \delta_{ij}}}})\;\;(Q_0)_{mi}(Q_0)_{mj} \text{(here not summing)},\; \forall i,j \in \{1,2,3\},$$
 we have
\begin{eqnarray}
&& \delta^2_{z^*z^*} J^*(Q_0,\tilde{\sigma}_0,z^*) \nonumber \\ &=&
\{ D_{ijkl}/K-(M_3)_{ijkl}-\overline{H}_{ijkl}\}
\nonumber \\ &\geq& \left\{ D_{ijkl}/K-\frac{(3/32K)^2 D_{ijkl}}{(K/2)^3}-\overline{H}_{ijkl}\right\} \nonumber \\ &\geq&
\{D_{ijkl}/(2K)-\overline{H}_{ijkl}\} \nonumber \\ &>& \mathbf{0},
\end{eqnarray}
$ \text{ in } \Omega, \;\forall z^* \in B^*(\tilde{\sigma}_0).$

From this, since $B^*(\tilde{\sigma}_0)$ is convex we obtain,
\begin{equation}\label{el20}\tilde{J}^*(Q_0,\tilde{\sigma}_0)=\inf_{ z^* \in B^*(\tilde{\sigma}_0)} J^*(Q_0,\tilde{\sigma}_0,z^*)= J^*(Q_0,\tilde{\sigma}_0,z^*_0).\end{equation}

Finally, from this, (\ref{el19}) and (\ref{el1}), we have,
\begin{eqnarray}
J(u_0)&=& \min_{u \in C \cap C_1} J(u) \nonumber \\ &=& \max_{(Q, \tilde{\sigma}) \in A^*} \tilde{J}^*(Q, \tilde{\sigma}) \nonumber \\
&=& \tilde{J}^*(Q_0,\tilde{\sigma}_0) \nonumber \\ &=& J^*(Q_0,\tilde{\sigma}_0,z_0^*).
\end{eqnarray}
The proof is complete.

\end{proof}

\section{Conclusion} In the present work, we have developed a duality principle  for non-linear elasticity.
We emphasize again the dual formulation obtained in concave and suitable for numerical computations.
Finally, we also highlight the results developed are applicable to other models in elasticity and other models of plates and shells.

\end{document}